\documentclass{article}

\usepackage{graphicx}
\usepackage{amsmath, amstext,mathbbol, amsthm, multirow, hyperref}%,fouridx,tensor,natbib}
\usepackage{cite}
\newtheorem{thm}{Theorem}[section]

\newtheorem{lemma}[thm]{Lemma}
\newtheorem{cond}{Condition}[section]

\newtheorem{remark}[thm]{Remark}
\newtheorem{example}{Example}[section]

\def\<<{``}
\def\R{\mathbb{R}}

\def\bE{\mathbf{E}}

\def\E{\mathbb{E}}

\def\cov{\mathrm{Cov}}

\usepackage{color}

\title{Higher Moments and Prediction Based Estimation for the COGARCH(1,1) model}
\author{Enrico Bibbona${}^{a}$, Ilia Negri${}^{b}$\\
${}^{a}$\small{University of Torino, Torino, Italy}\\
${}^{b}$\small{University of Bergamo, Bergamo, Italy}\\
\small{enrico.bibbona@unito.it}, \small{ilia.negri@unibg.it}\\
}

\begin{document}
\maketitle

\begin{abstract}
COGARCH models are continuous time version of the well known GARCH models of financial returns. The first aim of this paper is to show how the method of Prediction-Based Estimating Functions  (PBEFs) can be applied to draw statistical inference from observations of a COGARCH(1,1) model if the higher order structure of the process is clarified.
A second aim of the paper is to provide recursive expressions for the joint moments of any fixed order of the process. Asymptotic results are given and a simulation study shows that the method of PBEF outperforms the other available estimation methods.

\medskip 
Keywords: cogarch model, stochastic volatility models, prediction based estimating functions, parameter estimation, higher moments\end{abstract}

\section{Introduction}

The COGARCH model with order (1,1) has been introduced as a continuous version of the GARCH(1,1) model in \cite{kluppelbergJAP2004}. It is driven by a L\'evy  process $L=(L_{t})_{t\geq 0}$ through the equation $dG_{t}=\sigma_{t-} dL_{t}$ and the resulting volatility process $\sigma_{t}$ satisfies the stochastic differential equation $d \sigma^{2}_{t}=(\beta-\eta\sigma^{2}_{t-})dt + \phi\sigma^{2}_{t-} d[L]^{d}_{t}$
where $[L]^{d}_{t}$ is the discrete part of the quadratic variation of $L$.
Financial log-returns are modeled by the increments of the process $G_{t,h}=G_{t+h}-G_{t}$. 
The L\'evy process is the sole source of randomness and when it jumps both the price and the volatility jump at the same time.

For a more thorough presentation of such model, for the relation between GARCH sequences and the COGARCH process, for a comparison with other continuos time models with the same aim and for how this model is able to capture the stylized facts about financial data we refer the reader to the following papers \cite{kluppelbergJAP2004, kluppelbergReview, kluppelbergEJ2007, limitSPA, mleMaller, BuchmannLimit}. In the last few years many generalizations of the COGARCH model have been proposed. Among them, COGARCH processes of order (p,q) \cite{brockwellPQ} and multivariate COGARCH(1,1) \cite{stelzerMultivariate}.

A few methods for the estimation of the model parameters from a sample of equally spaced returns $G_{ir, r}=G_{(i+1)r}-G_{ir}$ are currently available.

In \cite{kluppelbergEJ2007} explicit estimators have been derived from a Method of Moments (MM). In \cite{mleMaller} a Pseudo Maximum Likelihood (PML) method has been  proposed that allows also for non equally spaced observations (see also \cite{mleSangyeol}), and in \cite{muller2010mcmc} an MCMC-based estimation method has been presented for the model driven by a compound Poisson process.

The first aim of the present paper is to demonstrate that the method of Prediction Based Estimating Functions (PBEFs) introduced in \cite{PredictionBased}  is applicable to the COGARCH(1,1) model and that its performance is better then the other available procedures. The general theory of PBEFs allows to find an optimal PBEF if the joint moments of the observations are explicitly known up to a certain order. 

Motivated by the search for an optimal PBEF, a second aim of the paper is to provide explicit expressions for the higher moments of the process. In particular a recursive formula for $\E\Big(G_{t}^{2i}\sigma^{2(k-i)}_{t}\Big)$ and $\bE_{v} \Big(G_{s,h}^{2i}\sigma^{2(k-i)}_{s+h}\Big)$ is found whenever they exist, for any total order $2k$ and any integer $i\leq k$ and for any $t,h>0$ and $s>v>0$. $\bE_{v}$ denotes conditional expectation with respect to the natural filtration ${\cal F}_{v}$.

Explicit expression for the joint moments $\E(G_{t_{h},r}^{2i_{h}}\,G_{t_{h-1},r}^{2i_{h-1}}\cdots\,G_{t_{2},r}^{2i_{2}}\,G_{t_{1},r}^{2i_{1}})$ are also provided  for any integers $i_{1}\cdots i_{h}$ and hence any total order $k=i_{1}+\cdots+ i_{h}$ and for  any times $t_{h}\cdots t_{1}$ such that $t_{i}-t_{i-1}\geq r$.

Up to the order four ($k=2$) our formulae coincide with those of \cite{kluppelbergEJ2007}, but explicit expressions for the higher orders are provided as a new result whose interest might go beyond the statistical methodology here proposed.

To validate the method, both asymptotic properties and finite-sample performances on a simulated dataset are investigated. The code for the numerical example and for the computation of the moments has been collected into an R \cite{R} package called COGARCH, which is briefly illustrated in the Supporting Information.

The paper is organized as follows. In Section \ref{model}  the definition and the properties of COGARCH(1,1) model are presented. In Section \ref{PBEF}, a suitable form of the prediction based estimation function method tailored for the COGARCH model is given, together with asymptotic results and the derivation of an optimal PBEF. In section \ref{moment}  the higher moments are derived and explicit formulae are given. In section \ref{explicitAssumptions} the necessary assumptions on COGARCH(1,1) needed to apply the method of PBEFs are stated and some example where these assumptions are satisfied are presented. Finally in Section \ref{simulation_study} finite-sample performances on a simulated dataset of the proposed method are investigated and compared with those of the other available methods. Supporting Information is available to describe the R package COGARCH and to provide a Mathematica \cite{Mathematica} notebook for the symbolic computation of the moments.

\section{The COGARCH(1,1) model} \label{model}
Let us introduce on a filtered probability space $(\Omega, {\cal F}, \{{\cal F}_t \}_{t\geq 0}, \mathbb P) $ with the usual properties,  a  L\'evy
process $L=\left( L_t\right)_{t\geq 0}$  with triplet $(\gamma,\tau^{2},\nu)$ and Poisson random measure $N$ (see \cite{applebaum, kyprianou ,protter}).
The COGARCH(1,1) model is defined as the solution $(G,\sigma^2)=\left(G_{t},\sigma^{2}_{t}\right)_{t\geq 0}$ of the following system of 
stochastic differential equations (SDE) driven by the L\'evy process $L$
\begin{equation}\left\{ \begin{aligned} & dG_{t}=\sigma_{t-} dL_{t}\\
&d \sigma^{2}_{t}=(\beta-\eta\sigma^{2}_{t-})dt + \phi\sigma^{2}_{t-} d[L]^{d}_{t},
\end{aligned}\right.\label{sde}\end{equation}
with initial value $G_0=0$ and $\sigma_0$ a random variable independent of the L\'evy process $\left( L_t\right)_{t\geq 0}$. 
The parameter space $\Theta\subset \R^3$ is defined as the set of those  $\theta=(\beta,\eta,\phi)$ such that $\beta>0$, $\eta> 0$ and $\phi> 0$.
By $[L]^{d}_{t}$,
 for every $t\geq 0$, we denote the discrete part of the quadratic variation
$[L]_{t}=\tau^{2}t +  [L]_{t}^{d}$ 
 of the driving L\'evy process $L_{t}$ defined as
\[
[L]^{d}_{t} =\int_{\mathbb{R}}x^{2}N(t,dx)=\sum_{0<s\leq t} (\Delta L_{s})^{2}
\]

with $\Delta L_{s}=L_{s}-L_{s-}$. 

The following is assumed throughout the paper.
\begin{cond}
\label{ass1}
$\E (L_{1})=0$ and $\E (L_{1}^{2})=1$. 
\end{cond}

If Conditions \ref{ass1} holds, then $L_{t}$ is a martingale and the volatility of the component $G_{t}$ is given solely by $\sigma_{t}$.

\begin{remark}
Under Condition \ref{ass1}, $\gamma$ and $\tau^2$ are not parameters of the model. Indeed since $\E (L_1)=0$,
$\gamma= \int_{|x|\geq1} x d\nu$. Moreover, since by the product formula $L_t^2 -[L]_{t}= 2\int_0^t L_s d L_s$, we have
\[  \E[L]_{1}=\tau^{2} +  \int_{\mathbb{R}}x^{2}\nu(dx)= \E\left(L_1^2 \right)=1,\]
hence $\tau^{2}= 1- \int_{\mathbb{R}}x^{2}\nu(dx)$.
Let  us however remark that the L\'evy measure $\nu$ may contain further parameters, that are supposed to be known.

\end{remark}

%questo semmai va nella intro
%
We list here without proof some properties of the COGARCH(1,1) that we will use later on.
%mixing, stationarity, Markov, time-homogeneity

The explicit solution of the second of equations \eqref{sde} with initial condition $\sigma^{2}_{u}$ at time $u$ is 
\begin{equation}\label{conditional}
\sigma^{2}_{t}=\beta\text{e}^{-(X_{t}-X_{u})}\int_{u}^{t}\text{e}^{-(X_{u}-X_{s})} \, ds + \text{e}^{-(X_{t}-X_{u})} \sigma^{2}_{u}.
\end{equation}
that is written in terms of the auxiliary process
\[X_{t}=\eta t- \sum_{0<s\leq t}\log\big(1+\phi (\Delta L_{s})^{2}\big)\]
whose Laplace transform can be written as
\[\E \text{e}^{-cX_{t}}=\text{e}^{t\Psi(c)}\]
for a function $\Psi$ defined as
\begin{equation}
\label{Psi}
\Psi(c)=-\eta c +\int_{R}\left[(1+\phi\, x^{2})^{c}-1\right] \nu(dx)=-\eta c+ \sum_{i=1}^{c} \binom{c}{i} \phi^{i} \int_{R}x^{2i}\nu(dx) .
\end{equation}
The Laplace transform is finite at $c$ if and only if $L_{1}$ has finite moments of order $2c$ and,  together with $\Psi(c)<0$, this is a sufficient condition for the process $\sigma^{2}_{t}$ to admit a stationary distribution (cf. \cite{kluppelbergJAP2004}) with finite moments of any order $k\leq c$  given by the following formula

\begin{equation}\E\sigma^{2k}_{\infty}=k!\beta^{k}\prod_{l=1}^{k}\frac{-1}{\Psi(l)}.\label{sigma_infinito}\end{equation}

In the COGARCH(1,1) model log-returns are represented as increments $G_{t,h}=G_{t+h}-G_t$ of the $G$ process.
The couple $\left(G_{t}, \sigma^2_{t}\right)_{t\geq 0}$ is a Markov process, but the single component $\left(G_{t}\right)_{t\geq 0}$  is not.
It can be proved (see \cite{kluppelbergJAP2004}) that if $\E (L^{4}_{1})<\infty$ and if the parameters are such that $\Psi(2)<0$, both the volatility process $(\sigma_t^2)_{t\geq 0}$ and the log-returns process $\left( G_{t,h}\right)_{t\geq 0}$ are stationary (allows for a stationary density) and strongly mixing with an exponentially decreasing rate.
We assume that $\sigma^{2}_0$ follows the stationary distribution.

\section{Prediction Based Estimating Functions}

\label{PBEF}
The statistical problem we address is the estimation of the parameter $\theta\in \Theta$ of a COGARCH(1,1) model whose driving L\'evy process is known a priori, from a sample of equally spaced log-returns $G_{jr, r}=G_{(j+1)r}-G_{jr}$, $j=1,\ldots, n$.
Three methods are currently available to this aim. Estimators based on the MM was introduced in \cite{kluppelbergEJ2007}. It is a very flexible tool that provides explicit estimators without the need for any assumption on the underlying L\'evy process. The estimators are consistent but not very efficient. A weak point of the method is that it is very sensitive to the tuning of a parameter. In \cite{mleMaller} a PML method is proposed. It is based on the approximation of the COGARCH model by a sequence of discrete-time GARCH series and on the application of Gaussian maximum likelihood to the GARCH approximation. A simulation study demonstrates that if the model is driven by a compound Poisson with normal jumps, using PML the mean squared errors is reduced with respect to the MM, but relevant biases can be found. Asymptotic results for this method are available in the high frequency asymptotic scheme in \cite{mleSangyeol}. MM and PML are the benchmarks against which we are going to test in Section \ref{simulation_study} the methodology introduced below.
Furthermore \cite{muller2010mcmc} introduced an MCMC-based estimation method in case the model is driven by a compound Poisson process whose applicability is limited also by its computational intensity.

The aim of this Section is to introduce the method of Prediction Based Estimating Functions (PBEFs) and to show that it can be applied to the COGARCH(1,1) model provided that we investigate further the structure of its moments. Asymptotic variances are also assessed in terms of such moments and an Optimal Prediction Based Estimating Function is found that minimizes such asymptotic variance.

%\subsection{Prediction Based Estimating Functions}

Estimating functions are functions of the parameters and of the data whose zeros are used as estimators of the parameters.  The most prominent example is the Score function which is known to be a martingale.
PBEFs were introduced by Sorensen  in \cite{PredictionBased} (see also \cite{ditlevsenpbe,PredictionBasedNew} for some recent developments) as a generalization of martingale estimation functions which are particularly suitable for non Markovian processes. 

The basic idea underlying PBEFs is that if you are able to predict the square $G^{2}_{ir, r}$ of the $i$-th observation of the returns on the basis of the previous $q$ squared observations and of the value of the parameters by means of a random variable $\pi^{i-1}(\theta, G_{(i-1)r, r},\cdots, G_{(i-q)r, r})$, denoted also by $\pi^{i-1}(\theta)$, then the value of $\theta$ that annihilates a weighted sum of the prediction errors $\sum_{i>q} w_{i}(G^{2}_{ir, r}-\pi^{i-1}(\theta))$ might be a good estimate of the unknown parameters.
To make such statement more formal let us introduce some notation and definitions.

Let ${\cal H}_i^\theta$ be the Hilbert space of all square integrable real functions of the observations $\{G_{jr, r}\}_{j=0}^{i}$ endowed with the usual inner product 
$$
\langle h , g \rangle= \E_\theta\left(h(G_{0,r},\ldots, G_{ir, r}) g(G_{0,r},\ldots, G_{ir, r})\right),
$$where $\E_\theta$ denotes the expectation under the model with parameter $\theta$. Let us fix an integer $q$. For any $i=q+1,\ldots,n$ we introduce the closed subspaces ${\cal P}^\theta_{i}$ of ${\cal H}_i^\theta$ spanned by 1 and the $q$ squared observations that come before the $i$-th, i.e. ${\cal P}^\theta_{i}=\text{span}(1,G_{(i-q) r, r}^2,\ldots,G_{(i-1) r, r}^2)$, where $1$ denotes the constant function with unit value. Provided  that 
 $\E_\theta(G_{ir, r}^{2})<\infty $ for every $\theta \in \Theta$ and every $i=1,\ldots, n$, we are interested in estimating functions of the form
\begin{equation}
S_n(\theta)=\sum_{i=q+1}^n w^{i-1}(\theta,n)(G^{2}_{ir, r} - \pi^{i-1}(\theta))
\label{estim}\end{equation}
which we call {\em prediction-based estimating functions}.
The vector $w^{i-1}(\theta,n)=(w_k^{i-1}(\theta,n))_{k=1}^3$ has components $w_k^{i-1}(\theta,n)\in {\cal P}^\theta_{i-1}$ and $\pi^{i-1}(\theta)$ is the minimum mean square error predictor of $G^{2}_{ir, r}$ in ${\cal P}^\theta_{i-1}$, that is the orthogonal projection of $G^{2}_{ir, r}$ on ${\cal P}^\theta_{i-1}$. Such projection exist and it is uniquely determined by the normal equations 
$$
\E_\theta\left(\pi (G^{2}_{ir, r}-\pi^{i-1}(\theta))\right)=0 \quad \forall \pi \in {\cal P}^\theta_{i-1}.
$$

Define 
$C(\theta)$ the covariance matrix of the $q$ vector $(G_{(i-1)r, r}^{2}, \ldots, G_{(i-q)r, r}^{2})^T$ and $b(\theta)$ the vector whose components are $b_j(\theta)=\cov_\theta(G_{(i-j) r, r}^{2}, G^{2}_{ir, r} )$ for $j=1,\ldots, q$.

As the increment process $G_{ir,r}$ is stationary the matrix  $C(\theta)$ and the vector $b(\theta)$ do not depend on $i$. We define the vector $a(\theta)=C(\theta)^{-1}b(\theta)$  whose components are denoted by $a_j(\theta)$ for $j=1,\ldots , q$ 
 and the scalar 
$a_0(\theta)=\E_\theta G^{2}_{ir, r}  - \sum_{j=1}^q a_j (\theta)\E_\theta(G_{(i-j)r, r}^{2})$. Moreover we denote by $\tilde{a}(\theta)$  the $q+1$ vector  $\tilde{a}(\theta)=(a_0(\theta),a_{1}(\theta),\ldots, a_{q}(\theta))^{T}$.

An explicit expression for the predictors is \cite{PredictionBased}
\begin{equation}\label{predictors}
\pi^{i-1}(\theta)=  a_0(\theta)+\sum_{j=1}^q a_j (\theta)G_{(i-j)r, r}^{2}
%a_0(\theta)+a(\theta)^T (G_{i-1-r, r}^{2}, \ldots, G_{i-q-r, r}^{2})=
\end{equation}
and it can be derived also from the Durbin-Levinson algorithm (cf. \cite{BrockwellDavis,PredictionBasedNew}).

%In our framework we deal with equally spaced observations $(G_i)_{i=1}^n$, giving return data $(G^{( r )}_{ir})_{i=1}^n$, where
%$$
%G^{( r )}_{ir}=G_{ir}-G_{r(i-1)}.
%$$
%We can put $Y_i=G_{ir}-G_{r(i-1)}$.
%The process $(Y_t)$ is stationary and as basis of ${\cal P}^\theta_{i-1}$ we choose 
%$Z^{i-1}_k=h_k(Y_{i-1}, Y_{i-2}, \ldots, Y_{i-s})$. For example we put $Z^{i-1}_k=Y_{i-k}^2$, for $k=1, \ldots q$ and $Z_0^{i-1}=1$.
%In such a case $a_0^{i-1}(\theta)$, the vector $a^{i-1}(\theta)$ and the matrix $C^{i-1}(\theta)$ do not depend on $i$ and  the estimating function, taking $f(Y_i)= Y_i^2$ can be written as
%$$
%L_n(\theta)=\sum_{i=q+1}^n \Pi^{i-1}(\theta)(Y_i^2 - a_0(\theta)-a_1(\theta) Y_{i-1}^{2}- \cdots, - a_q(\theta) Y_{i-q}^{2})
%$$
As the components $w_k^{i-1}(\theta,n)$ of the vector $w^{i-1}(\theta,n)$ are elements of ${\cal P}^\theta_{i-1}$, they can be decomposed as $w_k^{i-1}(\theta,n)=w^{i-1}_{k0}(\theta,n) + \sum_{j=1}^q
w^{i-1}_{kj}(\theta,n) G_{(i-j)r, r}^{2}$ for some scalars $w^{i-1}_{k0}(\theta,n)$ and $w^{i-1}_{kj}(\theta,n)$ $j=1,\ldots, q$ that we collect into the $p \times (q+1)$ matrices $W_{n}^{i-1}(\theta)$ whose elements are $w^{i-1}_{kl}(\theta,n)$ for $1\leq k\leq p$ and $0\leq l \leq q$.

With these notations the estimating function \eqref{estim} can be written as
$$
S_n(\theta)= \sum_{i=1}^n W_{n}^{i-1}(\theta) H^i(\theta)
$$
where  $H^{i}(\theta)$, $i=1,\ldots, n$, are $(q+1)-$vectors whose components are 
$$H^{i}_0(\theta) = G^{2}_{ir, r}-\- a_0(\theta)-a_1(\theta) G_{(i-1)r, r}^{2}- \cdots, - a_q(\theta) G_{(i-q)r, r}^{2}$$ 
and for  $k=1, \ldots, q$,
$$H^{i}_k(\theta)= G_{(i-k)r, r}^{2}(G^{2}_{ir, r}- a_0(\theta)-a_1(\theta)G_{(i-1)r, r}^{2} - \cdots - a_q(\theta)G_{(i-q)r, r}^{2}).$$ 
Since the increment process $G_{ir,r}$ is stationary, so is the vector $H^{i}(\theta)$ and there is no reason to give different weights for different $i$, thus we restrict our PBEFs to those that can be written in the form
\begin{equation}S_n(\theta)=  W_{n}(\theta)\sum_{i=1}^n H^i(\theta).
\label{estimating}\end{equation}

\begin{example}\label{esempio}\emph{
An intuitive example of estimating function is found minimizing the mean square prediction error
\begin{equation}
%\text{MSPE}
{\cal M}_n(\theta)=\sum_{i=q+1}^{n}\big(G^{2}_{ir, r}-\pi^{i-1}(\theta) \big)^{2}.\label{naive}
\end{equation}
An expression of the form \eqref{estimating} is found when searching for the critical points of \eqref{naive} by taking the derivatives. It gives for every $n$ the same weight matrix
\[W^{\text{MSPE}}(\theta)=(\partial_{\theta^T}\tilde{a}(\theta))^{T}.\]
%RIPRENDERE DOPO With this choice for the weights, the criterion \label{naive} or the corresponding estimating function can be evaluated whenever we know the expression of the covariance matrix $C$ which contains the joint moments of the process at different times up to the order four. 
}
\end{example}

\subsection{Asymptotic results}\label{as}
Let us introduce the vector $Z_{i}=(1,G_{(i-1)r, r}^{2}, \ldots, G_{(i-q)r, r}^{2})^T$, the matrix $\tilde C(\theta)= \E \big(Z^{i}(Z^{i})^{T}\big)$, and the matrix
$D(\theta)=-W(\theta)\tilde C(\theta)\partial_{\theta^T}\tilde{a}(\theta)$. In terms of such quantities we state the following conditions.

\begin{cond}
\begin{enumerate}
\item There exist a constant $\delta>0$ such that  $\E_\theta(G_1^{8+\delta})<\infty$.
\item The vector $\tilde{a}(\theta)$ and the matrix $W_{n}(\theta)$ are continuously differentiable with respect to $\theta$.
\item There exist a non-random matrix $W(\theta)$ such that for every compact set $K\subset\Theta$ 
\[W_n(\theta)\stackrel{\mathbb{P}_{\theta_{0}}}{\longrightarrow}W(\theta)\qquad\qquad  \partial_{\theta} W_n(\theta)\stackrel{\mathbb{P}_{\theta_{0}}}{\longrightarrow}\partial_{\theta}W(\theta)\]
uniformly for $\theta\in K$ as $n\longrightarrow\infty$.
\item The matrix  $D(\theta_{0})$ has full rank 3.
\item We have $W(\theta)\,\E_{\theta_{0}}\big(H^{i}(\theta)\big)\neq 0$ for any $\theta\neq \theta_{0}$.
\end{enumerate}\label{conditions}
\end{cond}

Since the increment process $G_{ir,r}$ is stationary and exponentially $\alpha$-mixing, denoting by $\mathcal{N}_{d}\left(\mu,\Sigma \right)$ a $d$-dimensional Gaussian random vector with mean $\mu$ e covariance matrix $\Sigma$, Theorem 4.3 in \cite{PredictionBasedNew} can be restated as follows.

\begin{thm} \label{asymptotics}
Assume that $\theta_{0}\in \text{int}\, \Theta$, and that Conditions \ref{conditions} are satisfied. Then a consistent estimator $\hat\theta_{n}$ exists that, with a probability tending to one as $n\rightarrow\infty$, solves the estimating equations $S_{n}(\hat\theta_{n})=0$ and it is unique in any compact $K\subseteq\Theta$ for which $\theta_{0}\in \text{int}\,K$. Moreover
\[\sqrt{n} \left(\hat{\theta}_{n}-\theta_{0}\right)\stackrel{\mathcal{D}}{\longrightarrow}\mathcal{N}_{3}\left(0,V(\theta_{0})\right)\]
with the matrix $V(\theta)$ given by
\[
V(\theta)=D^{-1}(\theta)W(\theta)M(\theta)W^{T}(\theta){D^{T}}^{-1}(\theta)
\]
and  $M(\theta)$ by
\[
M(\theta) = \E_\theta\big(H^{v}(\theta)H^{v}(\theta)^T \big)+\sum_{i=1}^{\infty} \Big[  \E_\theta\big(H^{v}(\theta)H^{v+i}(\theta)^T \big)+ \E_\theta\big(H^{v+i}(\theta)H^{v}(\theta)^T \big)\Big]\]
where the index $v$ can be  fixed to any value strictly greater than $q$ by stationarity.
\end{thm}

\begin{remark}
Knowing all simple and joint moments up to the order \emph{four} $\E\big(G_{jr, r}^{2} G^{2}_{ir, r}\big), \E\big(G_{jr, r}^{4}\big), \E\big(G_{jr, r}^{2}\big)$ for any integer $i,j$ is essential to calculate the predictors and hence to calculate any estimating function in the form \eqref{estimating}. Such explicit expressions for the COGARCH(1,1) model are given in \cite{kluppelbergEJ2007}.
However the asymptotic variance of the estimates involves the matrix $M$ which depends on all the simple and joint moments up to the order \emph{eight}, e.g. $\E\big(G_{jr, r}^{8}\big)$, $\E\big(G_{jr, r}^{6}\big)$, $\E\big(G_{jr, r}^{2}G^{6}_{ir, r}\big)$, $\E\big(G_{ir, r}^{2}G^{2}_{jr, r}G_{kr, r}^{2}G_{hr, r}^{2}\big)$ and similar.
Explicit expressions for such moments are currently not available and finding such expressions is the goal of Section \ref{moment}.
\end{remark}

\subsection{Optimal estimating functions}
\label{OPBE}
According to the general theory (cf. \cite{godambeheyde, PredictionBased,PredictionBasedNew} ), among all the PBEFs in the form \eqref{estimating} it is possible to select an \emph{optimal} one. 
The optimal PBEF will be such that the corresponding estimator has the smallest possible asymptotic variance.
The  weight matrix of the optimal PBEF is
\begin{equation}W^{\ast}_{n}=\partial_{\theta}\tilde{a}^{T}(\theta) \tilde C(\theta) M_{n}^{-1}(\theta)\label{optweights1}\end{equation}
where the matrix $M_{n}(\theta)$ given by
\begin{align}
&M_n(\theta) = \E_\theta(H^{q+1}(\theta)H^{q+1}(\theta)^T )+\notag\\
&\sum_{k=1}^{n-(q+1)}  \frac{n-q-k}{n-q}\left(  \E_\theta(H^{q+1}(\theta)H^{q+1+k}(\theta)^T )+ \E_\theta(H^{q+1+k}(\theta)H^{q+1}(\theta)^T )\right).\label{matriceMn}
\end{align}
Since for $n\rightarrow\infty$, $M_n(\theta) \rightarrow M(\theta)$, the matrix $M(\theta)$ can be used in the weights \eqref{optweights1} so that
\begin{equation}W^{\ast}=\partial_{\theta}\tilde{a}^{T}(\theta) \tilde C(\theta) M^{-1}(\theta)\label{optweights2}\end{equation}
and $W^{\ast}_{n}\rightarrow W^{\ast}$ as in the first requirements of Condition \ref{conditions}. The asymptotic variance of the corresponding  estimators is given by $V^\ast$, the inverse of the following matrix
\[{V^\ast}^{-1}=\partial_{\theta}\tilde{a}^{T}(\theta) \tilde C(\theta) M^{-1}(\theta)\tilde C(\theta)\partial_{\theta^{T}}\tilde{a}(\theta) \]

\begin{remark}
The optimal weight matrix $W^\ast$ depends on all the simple and joint moments up to the order \emph{eight}. Explicit expressions for such moments are currently not available and finding such expressions is the goal of  Section \ref{moment}.
\end{remark}

\begin{remark}
In term of the existence of higher moments, the condition requested for the asymptotic normality of the estimators obtained via PBEF and via the MM is the same ($\E_\theta(G_1^{8+\delta})<\infty$ for some $\delta>0$), see  \cite{kluppelbergEJ2007}.
\end{remark}

\begin{remark}
The calculation of the matrix $M_n(\theta)$ is computationally very demanding. This is even more relevant when, to find the zeros of the estimating function, it needs to be evaluated many times in a numerical optimization algorithm. However if the process is exponentially mixing (cf. \cite{PredictionBasedNew}), the convergence of $M_n(\theta)$ to $M(\theta)$ is very fast and a truncation of the sum in \eqref{matriceMn} that includes only the most relevant terms provides a good approximation.
\end{remark}

\section{Higher order moment structure}
\label{moment}
In this section we give conditions that assure the existence of simple and joint moments of the process $G_{t,r}$ up to any fixed order $k$, and we show how they can be calculated using an iterative procedure.  
\subsection{Notations}
Whenever we refer to the quadratic variation of the driving L\'evy process $L_{t}$ we denote it simply $[L]_{t}$. We reserve the less compact standard notation $[M,N]_{t}$ for the quadratic covariation of two semimartingales $M_{t}$ and $N_{t}$. Moreover, we often need to take quadratic covariations of quadratic variations and to this aim we introduce the following notation: quadratic variations of order $i+1$ for any $i\geq 1$ are defined as $[L]^{(i+1)}_{t}=\big[[L]^{(i)},[L]^{(1)}\big]_{t}$. With this notation if $[L]^{(1)}_{t}=L_{t}$, we have $[L]^{(2)}_{t}=[L]_{t}$, $[L]^{(3)}_{t}=\big[[L],L\big]_{t}$ and so on.

For $i>2$ the quadratic variations $[L]^{(i)}_{t}$ do not have any continuous component and we have  (cf. \cite{applebaum} Section 4.4.3),%since \[\begin{aligned}
%{}&\big[[L],L\big]_{t}=\big[[L]^{d},L\big]_{t}=\int_{\mathbb{R}}x^{3}N(t,dx)\\
%&\big[[L],[L]\big]_{t}=\big[[L]^{d},[L]\big]_{t}=\big[[L]^{d},[L]^{d}\big]_{t}=\int_{\mathbb{R}}x^{4}N(t,dx)\end{aligned}\] 
\[ [L]^{(i)}_{t}=\int_{\mathbb{R}}x^{i}N(t,dx) \qquad \E[L]^{(i)}_{t}=t\int_{\mathbb{R}}x^{i}\nu(dx).\]

However, in some iterative formula below where an index $i$ ranges between different values we will write $[L]^{d(i)}_{t}$ to keep track of the fact that when $i=2$ the object we mean is the discrete part of the quadratic variation. Let us also remark that for any $i\geq1$ and $j\geq1$ such that $i+j=k$, we have that $\big[[L]^{(i)},[L]^{(j)}\big]_{t}=[L]^{(k)}_{t}$.

\subsection{Higher moments of COGARCH(1,1)}
Let us start with two Lemmas that will be used repetitively  in the next sections.
\begin{lemma}
\label{LemmasigmaG}
Given a COGARCH(1,1) model \eqref{sde}, then for every integer $k$, it holds that
\begin{equation}
\label{sigma2k}
\sigma^{2k}_{t}=k\int_{0}^{t}\sigma^{2k-2}_{s-}(\beta-\eta\sigma^{2}_{s-})\,ds+ \sum_{i=1}^{k} \binom{k}{i}\phi^{i}\int_{0}^{t}\sigma^{2k}_{s-} d [L]^{d(2i)}_{s}
\end{equation}
and
\begin{equation}
G^{2k}_t=\sum_{i=1}^{2k} \binom{2k}{i}\int_{0}^{t}  G^{2k-i}_{s-}\sigma^{i}_{s-}d[L]_s^{(i)}\label{G2k}\end{equation}
\end{lemma}

\begin{proof}
We prove formula \eqref{sigma2k} by induction. For $k=1$ it is true.
Let us suppose \eqref{sigma2k} holds for $k-1$,  that is:
\[\sigma^{2k-2}_{t}=(k-1)\int_{0}^{t}\sigma^{2k-4}_{s-}(\beta-\eta\sigma^{2}_{s-})\,ds+ \sum_{i=1}^{k-1} \binom{k-1}{i}\phi^{i}\int_{0}^{t}\sigma^{2k-2}_{s-} d [L]^{d(2i)}_{s}.\]
Then by Ito product formula (see \cite{applebaum} Theorem 4.4.13), the \eqref{sde} and equation (4.15) in \cite{applebaum} p. 257 or Theorem 29 in \cite{protter} we obtain 
\[\begin{aligned}\sigma^{2k}_{t}&=\int_{0}^{t} \sigma^{2k-2}_{s-} d\sigma^{2}_{s}+\int_{0}^{t} \sigma^{2}_{s-} d\sigma^{2k-2}_{s}+ [\sigma^{2k-2},\sigma^{2}]_{t}=\\
&=k\int_{0}^{t}\sigma^{2k-2}_{s-}(\beta-\eta\sigma^{2}_{s-})\,ds + \phi \int_{0}^{t}\sigma^{2k}_{s-} d [L]^{d}_{s}+\\
&\phantom{=}+\sum_{i=1}^{k-1} \binom{k-1}{i}\phi^{i}\int_{0}^{t}\sigma^{2k}_{s-} d [L]^{d(2i)}_{s} + \sum_{i=1}^{k-1} \binom{k-1}{i}\phi^{i+1}\int_{0}^{t}\sigma^{2k}_{s-} d [L]^{d(2i+2)}_{s}=\\
%&=k\int_{0}^{t}\sigma^{2k-2}_{s}(\beta-\eta\sigma^{2}_{s})\,ds + k \phi \int_{0}^{t}\sigma^{2k}_{s} d [L]^{d}_{s}\\
%&\phantom{=}+ \sum_{i=2}^{k-1} \binom{k-1}{i}\phi^{i}\int_{0}^{t}\sigma^{2k}_{s} d [L]^{d(2i)}_{s} + \sum_{j=2}^{k} \binom{k-1}{j-1}\phi^{j}\int_{0}^{t}\sigma^{2k}_{s} d [L]^{d(2j)}_{s}\\
&=k\int_{0}^{t}\sigma^{2k-2}_{s-}(\beta-\eta\sigma^{2}_{s-})\,ds + k \phi \int_{0}^{t}\sigma^{2k}_{s-} d [L]^{d}_{s}\\
&\phantom{=}+ \sum_{i=2}^{k-1} \left[\binom{k-1}{i} +\binom{k-1}{i-1}\right]\phi^{i}\int_{0}^{t}\sigma^{2k}_{s-} d [L]^{d(2i)}_{s} + \phi^{k}\int_{0}^{t}\sigma^{2k}_{s-} d [L]^{d(2k)}_{s}.\\
\end{aligned}\]
The result follows by the well known Pascal's rule for the binomial coefficients. So equation \eqref{sigma2k} is proved.

The identity \eqref{G2k} was proved for $k=1$ and $k=2$ in \cite{kluppelbergEJ2007}. For any $k>2$ it follows by induction writing $G^{2k}_t$ as $G^{2}_tG^{2(k-1)}_t$ and applying Ito's product formula. Algebraic manipulations with repeated use of Pascal's rule are needed to simplify the coefficients.
\end{proof}

\begin{lemma}\label{3.2}
Let $k\geq 2$ be any fixed integer. Assume Condition \ref{ass1} holds. If $\E(L^{2k}_{1})<\infty$, $\Psi(k-1)<0$, and for any integer $2\leq c\leq k$, $\int_{\mathbb{R}}x^{2c-1}d\nu(x)=0$, then 
for every integer $ 1\leq i\leq k-1$ we have

%Soluzione[t_, S_, psi_] := 
 % Integrate[S[w]*Exp[-psi*w], {w, 0, t}]*Exp[psi*t];

\begin{equation}
\E(G_{t}^{2i}\sigma^{2(k-i)}_{t})= e^{t \Psi(k-i)}\int_0^t C_{ki}(s)e^{-s \Psi(k-i)}ds
\label{solEgs}
\end{equation}
%\Psi(k-i)\E\left(G_{t}^{2i}\sigma^{2(k-i)}_{t}\right)+C_{ki}(t),
where
\[\begin{aligned}C_{ki}(t)=&\beta(k-i)\E\left(G_{t}^{2i}\sigma^{2(k-i)-2}_{t}\right)+\sum_{j=1}^{i} \binom{2i}{2j}  \E\left(G^{2i-2j}_{t}\sigma^{2(k-i)+2j}_{t}\right)\E\left([L]_{1}^{(2j)}\right)\\
&+\sum_{j=1}^{k-i} \binom{k-i}{j} \sum_{h=1}^{i} \binom{2i}{2h} \phi^{j} \E\left([L]_{1}^{(2j+2h)}\right)
\E\left(G^{2i-2h}_{t} \sigma^{2(k-i)+2h}_{t}\right).
\end{aligned}\]
\end{lemma}
\begin{proof}
Indeed, for the Ito product formula and for Lemma  \ref{LemmasigmaG} we can write
\begin{align}&G_{t}^{2i}\sigma^{2(k-i)}_{t}=\int_{0}^{t} G_{s-}^{2i}d\sigma^{2(k-i)}_{s} +\int_{0}^{t} \sigma^{2(k-i)}_{s-} dG_{s}^{2i} +[G^{2i},\sigma^{2(k-i)}]_{t}\notag\\
& =(k-i)\int_{0}^{t}G_{s-}^{2i}\sigma^{2(k-i)-2}_{s-}(\beta-\eta\sigma^{2}_{s-})\,ds+ \sum_{j=1}^{k-i} \binom{k-i}{j}\phi^{j}\int_{0}^{t}G_{s-}^{2i}\sigma^{2(k-i)}_{s-} d [L]^{d(2j)}_{s}\notag\\
&\phantom{m}+\sum_{j=1}^{2i} \binom{2i}{j}\int_{0}^{t}  G^{2i-j}_{s-}\sigma^{2(k-i)+j}_{s-}d[L]_s^{(j)}+\notag\\
&\phantom{m}+\left[\sum_{j=1}^{k-i} \binom{k-i}{j}\phi^{j}\int_{0}^{t}\sigma^{2(k-i)}_{s-} d [L]^{d(2j)}_{s}, \sum_{h=1}^{2i} \binom{2i}{h}\int_{0}^{t}  G^{2i-h}_{s-}\sigma^{h}_{s-}d[L]_s^{(h)}\right]\notag\\
&=-(k-i)\eta\int_{0}^{t}G_{s-}^{2i}\sigma^{2(k-i)}_{s-}\,ds+ \sum_{j=1}^{k-i} \binom{k-i}{j}\phi^{j}\int_{0}^{t}G_{s-}^{2i}\sigma^{2(k-i)}_{s-} d [L]^{d(2j)}_{s}\notag\\
&+\beta (k-i) \int_{0}^{t}G_{s-}^{2i}\sigma^{2(k-i)-2}_{s-}\,ds\phantom{m}+\sum_{j=1}^{2i} \binom{2i}{j}\int_{0}^{t}  G^{2i-j}_{s-}\sigma^{2(k-i)+j}_{s-}d[L]_s^{(j)}+\label{Gs}\\
&\phantom{m}+\sum_{j=1}^{k-i} \binom{k-i}{j} \sum_{h=1}^{2i} \binom{2i}{h} \phi^{j}\int_{0}^{t}\sigma^{2(k-i)+h}_{s-}  G^{2i-h}_{s-}d[L]_s^{(2j+h)}.\notag
\end{align}

Observe that both $G_{s-}=G_{s}$ and $\sigma^{2}_{s-}=\sigma^{2}_{s}$ almost surely. Taking the expectation, applying the compensation formula (see for example \cite[Theorem 4.4]{kyprianou}), differentiating with respect to $t$, and remembering \eqref{Psi} and that for any integer $2\leq c\leq k$, $\int_{\mathbb{R}}x^{2c-1}d\nu(x)=0$, we obtain
\[\begin{aligned}&\frac{d}{dt}\E\left(G_{t}^{2i}\sigma^{2(k-i)}_{t}\right)=\\
& =\Psi(k-1)\E\left(G_{t}^{2i}\sigma^{2(k-i)}_{t}\right)+\\
&\phantom{m}+\beta(k-i)\E\left(G_{t}^{2i}\sigma^{2(k-i)-2}_{t}\right)+\sum_{j=1}^{i} \binom{2i}{2j}  \E\left(G^{2i-2j}_{t}\sigma^{2(k-i)+2j}_{t}\right)\E\left([L]_{1}^{(2j)}\right)+\\
&\phantom{m}+\sum_{j=1}^{k-i} \binom{k-i}{j} \sum_{h=1}^{i} \binom{2i}{2h} \phi^{j} \E\left(G^{2i-2h}_{t} \sigma^{2(k-i)+2h}_{t}\right) \int_{\mathbb{R}} x^{(2j+2h)} d\nu(x)
\end{aligned}\]

Simplifying we obtain  
\[\frac{d}{dt}\E\left(G_{t}^{2i}\sigma^{2(k-i)}_{t}\right)=\Psi(k-i)\E\left(G_{t}^{2i}\sigma^{2(k-i)}_{t}\right)+C_{ki}(t),\]
and a stationary solution of this ode with initial condition $\E\left(G_{0}^{2i}\sigma^{2(k-i)}_{0}\right)=0$  exists if $\Psi(k-i)<0$ and is given by \eqref{solEgs}. 
This completes the proof. 
\end{proof}

\begin{thm}
Let $k\geq 1$ be any fixed integer. Assume Condition \ref{ass1} holds.  If $\Psi ( k) <0$, $\E(L_{1}^{2k}) <\infty$, and, for every integer $1\leq i\leq k$, 
$\int_{\mathbb{R}}x^{2i-1}d\nu(x)=0$, then 
\vspace{-2mm}
$$%\begin{equation}
\E G^{2k}_t=\sum_{i=1}^{k} \binom{2k}{2i} \E([L]_{1}^{(2i)}) \int_{0}^{t} \E \big(G^{2k-2i}_{s}\sigma^{2i}_{s}\big)ds. 
%\label{EG2k}
$$%\end{equation}
\end{thm}

\begin{proof} The result follows from \eqref{G2k} by taking the expectation and applying the compensation formula \cite[Theorem 4.4]{kyprianou}). Note that
$\E([L]^{2i-1})=\int_{\R} x^{(2i-1)} d\nu(x) =0$ for every integer $2\leq i\leq k$.
%The contribution of the first summand is then singled out taking in account that $\E[L]_t= \E L_t^2=1$. 
\end{proof}

\subsection{Higher conditional moments}
Conditional moments of the product are necessary not only to derive joint moments of higher order of the log returns, as we will do in the next section, but could be useful by itself and for this reason the result  is presented in this section. Let us remind that $\bE_{v}$ denotes conditional expectation with respect to the natural filtration ${\cal F}_{v}$.

\begin{thm} \label{att_condizionali}
For every $k$ and for any $0\leq i\leq  k$, for $h>0$, $s>0$  and given $0<v< s$, we have
\begin{equation}
\bE_{v} \left[G_{s,h}^{2i}\sigma^{2(k-i)}_{s+h}\right]=\sum_{r=0}^{k} J_{kir}(h,s-v)\, \sigma^{2r}_{v}
\label{G2isigma2k-i}
\end{equation}
where $G_{s,h}^{2i}=(G_{s+h}-G_s )^{2i}$ and the coefficients $J_{kir}(h,d)$ are deterministic and can be calculated recursively as follows. 

First
\begin{equation}
J_{k0k}(h,d)=\text{e}^{(h+d)\Psi(k)},
\label{Jk0k}
\end{equation}
then for every $1\leq i\leq k$
\begin{equation}
\begin{aligned}J_{k0(k-i)} (h,d)=&\frac{k!}{(k-i)!}  \beta^{i} \int_{0}^{h+d}ds_{i}\int^{s_{i}}_{0}ds_{i-1}\cdots\\
&\cdots\int_{0}^{s_{2}}\text{e}^{(h+d-s_{i})\Psi(k)+(s_{i}-s_{i-1})\Psi(k-1)+\cdots +s_{1}\Psi(k-i)} ds_{1}.\end{aligned}
\label{Jk0k-i}
\end{equation}
%moreover
%\[\begin{aligned}J_{k00} (h,d)=& \beta^{k} k! \int^{h+d}_{0}ds_{k}\int_{0}^{s_{k}}ds_{k-1}\cdots\\
%&\dots\int_{0}^{s_{2}}\text{e}^{(h+d-s_{k})\Psi(k)+(s_{k}-s_{k-1})\Psi(k-1)\cdots (s_{2}-s_{1})\Psi(1)} ds_{1}.\end{aligned}\]

\medskip

For any fixed $k$ and $i<k$ the coefficients  $J_{kir}(h,d)$ can be derived as follows

%\begin{equation}
\begin{align}J_{kik}(h,d) &= \text{e}^{h\Psi(k-i)}\int_{0}^{h} \text{e}^{-w\Psi(k-i)}\left[ \sum_{j=1}^{i} \E\left([L]_{1}^{(2j)}\right) \binom{2i}{2j} J_{k(i-j)k}(w,d)+\right. \notag\\
&+\left.\sum_{m=1}^{i}\binom{2i}{2m}J_{k(i-m)k} (w,d)\sum_{j=1}^{k-i}\binom{k-i}{j}\phi^{j} \E\left([L]_{1}^{(2j+2m)}\right)\right]\,dw. \phantom{accipicchio}\label{Jkik}
\raisetag{1.5cm}\end{align}

%\end{equation}

For any $r<k$

\begin{align}J_{kir}(h,d)&= \text{e}^{h\Psi(k-i)}\int_{0}^{h} \text{e}^{-w\Psi(k-i)}\Bigg[(k-i) \beta J_{(k-1)i r}(w,d)+\notag \\
&+ \sum_{j=1}^{i} \E\left([L]_{1}^{(2j)}\right)\binom{2i}{2j} J_{k(i-j)r}(w,d)+
\label{Jkir}
\\
&+\left.\sum_{m=1}^{i}\binom{2i}{2m}J_{k(i-m)r} (w,d)\sum_{j=1}^{k-i}\binom{k-i}{j}\phi^{j}\E\left([L]_{1}^{(2j+2m)}\right)\right]\,dw. \notag
 \end{align}

Finally, for any $r\leq k$ we have 
$$%\begin{equation}
J_{kkr}(h,d)= \sum_{j=0}^{k-1} \binom{2k}{2(k-j)}\E [L]_1^{(2(k-j))}  \int_{0}^{h}J_{kjr}(u,d) du. %\label{Jkkr}
$$%\end{equation}

\end{thm}

\begin{proof}
Fix $k$. Let us start to prove equation \eqref{G2isigma2k-i} for $i=0$. To calculate 
$
\bE_{v} \left(\sigma^{2k}_{t}\right)
$ we apply formula \eqref{conditional} with initial condition at time $v$.
%\begin{equation}\label{conditional}
%\sigma^{2}_{t}=\beta\text{e}^{-(X_{t}-X_{u})}\int_{u}^{t}\text{e}^{-(X_{u}-X_{s})} \, ds + \text{e}^{-(X_{t}-X_{u})} \sigma^{2}_{u}.
%\end{equation}
%Taking conditional expectation of \eqref{conditional} with respect to $\mathcal{F}_{u}$ we get
%\begin{equation}
%\E_{u} \sigma^{2}_{t}= \text{e}^{(t-u)\Psi(1)} \sigma^{2}_{u}+\frac{\beta}{\Psi(1)}(\text{e}^{(t-u)\Psi(1)}-1) =S_{11}(t-u)\sigma^{2}_{u}+S_{10}(t-u)
%\end{equation}
For every $v\leq s_{1}\leq\cdots\leq s_{k}\leq t$, we have
$$
\begin{aligned}
\sigma^{2k}_{t}&=\left(\beta\int_{v}^{t}\text{e}^{-(X_{t}-X_{s})} \, ds + \text{e}^{-(X_{t}-X_{v})} \sigma^{2}_{v}\right)^{k}\\
&= \text{e}^{-k(X_{t}-X_{v})} \sigma^{2k}_{v}\\
&+\sum_{i=1}^{k} \binom{k}{i}  \text{e}^{-(k-i)(X_{t}-X_{v})} \sigma^{2(k-i)}_{v} \beta^{i} \int_{v}^{t}\text{e}^{-(X_{t}-X_{s_{1}})} \, ds_{1}\cdots \int_{v}^{t}\text{e}^{-(X_{t}-X_{s_{i}})} \, ds_{i}\\
&=\text{e}^{-k(X_{t}-X_{v})} \sigma^{2k}_{v}+\sum_{j=1}^{k} \binom{k}{j} \sigma^{2(k-j)}_{v} \beta^{j} \cdot j!\cdot\\
&\int_{v}^{t}ds_{j}\int^{s_{j}}_{v}ds_{j-1}\cdots\int_{v}^{s_{2}}\text{e}^{-k(X_{t}-X_{s_{j}})}\text{e}^{-(k-1)(X_{s_{j}}-X_{s_{j-1}})} \cdots \text{e}^{-(k-j)(X_{s_{1}}-X_{v})}ds_{1}.
\end{aligned}
$$
The increments $X_{t}-X_{v}$ are independent of ${\cal F}_v$ and of $\sigma_v^2$ which is ${\cal F}_v$-measurable. Time homogeneity of $X_{t}$ ensures that $X_{t}-X_{v}\stackrel{D}{=}X_{t-v}$. Then taking the conditional expectation with respect to ${\cal F}_v$,  by equation \eqref{Psi}, we get
$$
\begin{aligned}
\bE_{v}\left(\sigma^{2k}_{t}\right)&=
\text{e}^{(t-v) \Psi(k)} \sigma^{2k}_{v}+\sum_{j=1}^{k} \binom{k}{j} \sigma^{2(k-j)}_{v} \beta^{j} \cdot j! \cdot\\
&\cdot \int_{v}^{t}ds_{j}\int^{s_{j}}_{v}ds_{j-1}\cdots\int_{v}^{s_{2}}\text{e}^{(t-s_{j})\Psi(k)} \text{e}^{(s_{j}-s_{j-1})\Psi(k-1)}  \cdots \text{e}^{(s_{1}-v)\Psi(k-j)} ds_{1}.
\end{aligned}
$$
%\sum_{r=0}^{k}J_{k0r}(t-s, s-v)\,\sigma_{v}^{2r}\label{sigmaCond}\end{equation}
that gives the thesis once observed that the coefficients in \eqref{G2isigma2k-i} with $i=0$ are actually dependent only on
the sum of their arguments. 
%Let us remark that for $t\rightarrow v$, for any $v< s<t$,  $J_{k0r}(t-s, s-v)$ tends to 0 for every $r\neq k$, while $J_{k0k}(t-s, s-v)\rightarrow 1$, as expected. 

Now let us prove equation \eqref{G2isigma2k-i} for $i=k=1$. By the Ito product formula

%{G_{s, s+h}^{(h)}}^{2}
\begin{align} G_{s,h}^{2}=&(G_{s+h}-G_s )^{2}=\left( \int _s^{s+h} \sigma_{u-} dL_u\right)^2 =\notag \\
=&2\int_{0}^{h}(G_{(s+u)-}-G_{s-})\sigma_{s+u}dL_{s+u}+\int_{0}^{h}\sigma^{2}_{s+u} d [L]_{s+u}\notag\end{align}
and by the compensation formula for the conditional expectation \cite[Corollary 4.5]{kyprianou}
\[\bE_r G_{s,h}^{2}= \E([L]_{1}) \int_{0}^{h}\bE_r(\sigma^{2}_{s+u}) d u.\]
Now let us assume as inductive hypothesis  that \eqref{G2isigma2k-i} holds for any given integer value $k\leq a-1$ and for all
$i\leq k$. We have to show that it holds also for $k=a$ and all $i\leq k$. Let us start to notice that for $k=a$ and $i=0$ this has already been proved. So it is enough to prove that equation  \eqref{G2isigma2k-i} for $k=a$ and all $i\leq b-1<a$ implies 
 equation  \eqref{G2isigma2k-i} with $k=a$ and $i=b$. 
 For every $k$, by writing $G_{s,h}^{2k}= G_{s,h}^{2(k-1)}G_{s,h}^{2}$ and applying the Ito product formula, we have in analogy with \eqref{G2k}
\begin{equation}G_{s,h}^{2k}=\sum_{i=1}^{2k} \binom{2k}{i}\int_{0}^{h}  (G_{(s+u)-}-G_{s})^{2k-i}\sigma^{i}_{s+u}d[L]_{s+u}^{(i)}.\label{G2kdoppia}\end{equation}
With the analogous calculations that lead to formula \eqref{Gs},  Ito product formula guarantees that (denoting the increment $G_{(s+h)-}-G_{s}$ by $G_{s,h-}$) 
 \begin{align}G_{s,h}^{2b}\sigma^{2(a-b)}_{s+h}&=-(a-b)\eta\int_{0}^{h}G_{s,u-}^{2b}\sigma^{2(a-b)}_{(s+u)-}\,du+ \sum_{j=1}^{a-b} \binom{a-b}{j}\phi^{j}\int_{0}^{h}G_{s,u-}^{2b}\sigma^{2(a-b)}_{(s+u)-} d [L]^{d(2j)}_{s+u}\notag\\
&+\beta (a-b) \int_{0}^{h}G_{s,u-}^{2b}\sigma^{2(a-b-1)}_{(s+u)-}\,du\phantom{m}+\sum_{j=1}^{2b} \binom{2b}{j}\int_{0}^{h}  G^{2b-j}_{s,u-}\sigma^{2(a-b)+j}_{(s+u)-}d[L]_{s+u}^{(j)}+\notag\\
&\phantom{m}+\sum_{j=1}^{(a-b)} \binom{a-b}{j} \sum_{h=1}^{2b} \binom{2b}{h} \phi^{j}\int_{0}^{h} G^{2b-h}_{s,u-} \sigma^{2(a-b)+h}_{(s+u)-} d[L]_{s+u}^{(2j+h)}.\notag %\label{Gdoppios}
\end{align}
 
again analogously to the proof of Lemma \ref{3.2} we can show that for $v<s$ $\bE_{v}\left(G_{s,h}^{2b}\sigma^{2(a-b)}_{s+h}\right)$ solves the following ode

\[\frac{d}{dh}\bE_{v}\left(G_{s,h}^{2b}\sigma^{2(a-b)}_{s+h}\right)=\Psi(a-1)\bE_{v}\left(G_{s,h}^{2b}\sigma^{2(a-b)}_{s+h}\right)+C_{a b}(h,s,v)\]
with
\[\begin{aligned}C_{ab}(h,s,v)=&(a-b) \beta\bE_{v}\left(G_{s,h}^{2b}\sigma^{2(a-b-1)}_{s+h}\right)+\sum_{j=1}^{b} \binom{2b}{2j}  \bE_{v}\left(G_{s,h}^{2(b-j)}\sigma^{2(a-b+j)}_{s+h}\right)\E\left([L]_{1}^{(2j)}\right)+\\
&+ \sum_{m=1}^{b} \binom{2b}{2m}  \bE_{v}\left(G_{s,h}^{2(b-m)} \sigma^{2(a-b+m)}_{s+h}\right) \sum_{j=1}^{a-b} \binom{a-b}{j}\phi^{j}\E\left([L]_{1}^{(2j+2m)}\right)
\end{aligned}\]
with initial condition $\bE_{v}\left(G_{s,0}^{2b}\sigma^{2(a-b)}_{s}\right)=0$.
Solving the ode we get

\begin{equation}
\bE_{v}(G_{s,h}^{2b}\sigma^{2(a-b)}_{s+h})= e^{h \Psi(a-b)}\int_0^h C_{a b}(u,s,v)e^{-u \Psi(a-b)}du
\label{solEvGS}
\end{equation}

Let us now observe that by the inductive hypothesis formula \eqref{G2isigma2k-i} is true for all the conditional expectations appearing in $C_{ab}(u,s,v)$, thus
\[\begin{aligned} {}&\bE_{v}\left(G_{s,h}^{2b}\sigma^{2(a-b-1)}_{s+h}\right)=\sum_{r=0}^{a-1} J_{(a-1)b r}(h,s-v)\sigma^{2r}_{v}\\
&\bE_{v}\left(G_{s,h}^{2(b-j)}\sigma^{2(a-b+j)}_{s+h}\right)=\sum_{r=0}^{a} J_{a (b-j) r}(h,s-v)\sigma^{2r}_{v}\\
&\bE_{v}\left(G_{s,h}^{2(b-m)} \sigma^{2(a-b+m)}_{s+h}\right) = \sum_{r=0}^{a} J_{a(b-m)r}(h,s-v)\sigma^{2r}_{v}.
\end{aligned}\]
Substituting in \eqref{solEvGS} we get that $\bE_{v}(G_{s,h}^{2b}\sigma^{2(a-b)}_{s+h})$ is itself a polynomial in $\sigma_{v}^{2}$ of highest order $a$ with coefficients given by  formula \eqref{Jkir} if $r\neq k$ of according to formula \eqref{Jkik} if $r=k$.

To conclude the proof we need to show that if \eqref{G2isigma2k-i} is true for $k=a$ and $i\leq a-1$ then it is also true for $k=i=a$. To this aim we rewrite \eqref{G2kdoppia}, with $k=a$ 
 
\[G_{s,h}^{2a}=\sum_{i=1}^{2a} \binom{2a}{i}\int_{s}^{s+h}  (G_{u-}-G_{s})^{2a-i}\sigma^{i}_{u-}d[L]_u^{(i)}.\]
Redefining the index of the sum as $j=a-i$ we have for all $v<s$ and $h>0$

$$\begin{aligned}
\bE_{v} \left(G_{s,h}^{2a} \right)&=\sum_{j=0}^{a-1} \binom{2a}{2(a-j)}\E [L]_1^{(2(a-j))}\int_{s}^{s+h} \bE_{v} \left[G_{s,h}^{2j} \sigma^{2(a-j)}_{u}\right]du \\
&=\sum_{r=1}^{a} \sigma^{2r}_{v} \left(\sum_{j=0}^{a-1} \binom{2a}{2(a-j)}\E [L]_1^{(2(a-j))}  \int_{s}^{s+h}J_{ajr}(h,s-v) du \right) 
\end{aligned}$$
hence

\[J_{aaj}(h,s-v)= \sum_{j=0}^{a-1} \binom{2a}{2(a-j)}\E [L]_1^{(2(a-j))}  \int_{s}^{s+h}J_{ajr}(h,s-v) du \]

\end{proof}

\begin{remark}
In the coefficients given by \eqref{Jk0k} and \eqref{Jk0k-i} the dependence from the time lags $h$ and $d$ came just through the total time lag $h+d$.

\end{remark}

\subsection{Joint Moments}

In view of the construction of estimators based on the method of PBEF, the following theorem provides the result we need.
\begin{thm}
\label{congiunti}Fix any integer $k\geq 1$. Let $\Psi ( k) <0$, $\E(L_{1}^{2k}) <\infty$ and for every $c\leq k$ let  $\E([L]^{2c-1})=\int_{\R} x^{(2c-1)} d\nu(x) =0$.  For any integer $h\geq 2$ and any set of integers $i_j\geq0$, $j=1, \ldots, h$ such that $i_1+i_2 +\ldots +i_h=k$ we have for every $0\leq t_1 <  t_2 <\ldots < t_h <T$, and $t_j>t_{j-1}+r$ for any $j$,  
$$\begin{aligned}
{}&\E\left(G_{t_{h},r}^{2i_{h}}\,G_{t_{h-1},r}^{2i_{h-1}}\cdots\,G_{t_{2},r}^{2i_{2}}\,G_{t_{1},r}^{2i_{1}}\right)=\sum_{r_{1}=0}^{i_{h}} \Bigg(J_{i_{h}i_{h}r_{1}}(r,s_{h-1}-s_{h-2}-r )\\
&\phantom{cazzo merda}\cdot \sum_{r_{2}=0}^{r_{1}+i_{h-1}}\bigg\{J_{(r_{1}+i_{h-1})i_{h-1}r_{2}}(r,s_{h-2}-s_{h-3}-r) \\
&\phantom{cazzo merda}\cdot  \sum_{r_{3}=0}^{r_{2}+i_{h-2}}\bigg[J_{(r_{2}+i_{h-2})i_{h-2}r_{3}}(r,s_{h-3}-s_{h-4}-r)\cdots\\
&\phantom{cazzo merda} \cdots\sum_{r_{h-1}=0}^{r_{h-2}+i_2}\bigg(J_{(r_{h-2}+i_2)i_2r_{h-1}}(r,s_1-r ) 
\E\left(\sigma^{2r_{h-1}}_r G_r^{2i_1}\right)\bigg)\bigg]\bigg\}\Bigg),
\end{aligned}$$
where $s_{j-1}=t_j-t_1$, $j=h,\ldots, 2$.
\end{thm}

\begin{proof}
By stationarity of $G_{t,r}$ we can write
$$
\E\left(G_{t_{h},r}^{2i_{h}}\,G_{t_{h-1},r}^{2i_{h-1}}\cdots\,G_{t_{2},r}^{2i_{2}}\,G_{t_{1},r}^{2i_{1}}\right)=\E\left(G_{s_{h-1},r}^{2i_{h}}G_{s_{h-2},r}^{2i_{h-1}}\cdots\,G_{s_{1},r}^{2i_{2}}\,G_{r}^{2i_{1}}\right)
$$
Taking the conditional expectation repeatedly in the right hand side we get
$$
\E\Bigg[\bE_{r}\left\{\cdots\bE_{s_{h-3}+r}\left[\bE_{s_{h-2}+r}\left(G_{s_{h-1},r}^{2i_{h}}\right)\,G_{s_{h-2},r}^{2i_{h-1}}\right]G_{s_{h-3},r}^{2i_{h-2}}\cdots G_{s_{1},r}^{2i_{2}}\right\}G_{r}^{2i_{1}}
\Bigg].
$$
Starting with the innermost conditional expectation we can apply Theorem \ref{att_condizionali}  reducing the argument to a deterministic part that just contains some $J$ coefficient (with appropriate indexes) and a part that is measurable with respect to the $\sigma$-algebra we are conditioning on (${\cal F}_{s_{h-2}+r}$ in that case). 
Repeating the procedure for every conditional expectation subsequently from the innermost to the outermost we are left we get the thesis.  
\end{proof}

\begin{remark}
The condition $\int_{\R} x^{(2c-1)} d\nu(x) =0$ for any $2\leq c\leq k$ is rather strong and it is close to assume that the L\'evy measure is symmetric. A more detailed discussion of its relevance in the model is deferred to Section \ref{discussionAssumptions} 
\end{remark}

\begin{remark}[Explicit expressions]
For $k=1$ and $k=2$  the moments were already calculated (see formulae (9) and (10) in \cite{kluppelbergEJ2007}). We recover equivalent expressions. Explicit formulae for all the joint moments with $k=3$ and $k=4$ have been derived performing the iterations by symbolic computation in Mathematica \cite{Mathematica}. The final results are very long and it is unfeasible to display them here. We include as supporting information to the paper (see also the Section at the end of the paper) a Mathematica notebook \cite{Mathematica}, that allows to calculate and manipulate all the expressions of the moments and that is able to produce as output a computer-readable form that can be evaluated numerically in R \cite{R} or in C \cite{CKR}. Functions that allow the numerical evaluation of the joint moments are available from the R package COGARCH that is described in a dedicated file included in the Supporting Information.
\end{remark}

\section{The estimation method and the assumptions} \label{explicitAssumptions}
We are now able to summarize how the iterative expressions found in the previous section  can be concretely turned into a feasible estimation method. In particular we choose two PBEFs: the one corresponding to the minimum mean squared prediction error (MSPE) of Example \ref{esempio} and the optimal one with weights \eqref{optweights1}. We want to spell out the conditions we need in each case to get both calculable expressions and nice asymptotic properties.

\subsection{MSPE estimation}
The estimation by means of \eqref{naive} requires the explicit expressions of the moments up to the order four in order to compute the predictors \eqref{predictors} and the contrast function \eqref{naive} that is then minimized numerically.

A set of conditions that both provide the good asymptotic properties of Theorem \ref{asymptotics} and the necessary moments are listed below.

\begin{cond}
\begin{enumerate}
\item $\E (L_{1})=0$ and $\E (L_{1}^{2})=1$. 
\item There exist a constant $\delta>0$ such that  $\E_\theta(L_1^{8+\delta})<\infty$.
\item $\Psi(4)<0$.
\item $\int_{\mathbb{R}}x^{3}d\nu(x)=0$
\end{enumerate}\label{final_cond_MSPE}
\end{cond}

\begin{remark}
The MSPE method does not need any additional condition with respect to those that were asked for in order to get similar asymptotic properties for the estimators derived by a method of moments in \cite{kluppelbergEJ2007}.
\end{remark}
\subsection{Optimal PBEF}
The OPBEF introduced in Section \ref{OPBE} requires the explicit expressions of the moments up to the order eight in order to compute the weights \eqref{optweights1}.
The existence of such moments is already guaranteed by Conditions \ref{final_cond_MSPE}, but to calculate their explicit expressions (cf. Theorem \ref{congiunti}) we need the following further assumptions.

\begin{cond}\label{further_conditions_OPBE}
$\int_{\mathbb{R}}x^{5}d\nu(x)=0$ and $\int_{\mathbb{R}}x^{7}d\nu(x)=0$.
\end{cond}

Under Conditions \ref{final_cond_MSPE} and \ref{further_conditions_OPBE} the estimating function exists and the estimators are consistent and asymptotically normal according to Theorem \ref{asymptotics}.

\subsection{Discussion on the assumpions}\label{discussionAssumptions}
Two of the previous conditions are rather strong: the existence of (slightly more than) eight moments for the L\'evy process and 
the condition $\int_{\R} x^{(2c-1)} d\nu(x) =0$ for any $2\leq c\leq k$, that is close to assume that the L\'evy measure is symmetric.
The former is necessary if we want to prove the asymptotic normality of nearly any estimator, since (together with $\Psi(4)<0$) it guarantees the mixing condition for the process of log returns (cf. also the paper \cite{kluppelbergEJ2007} on the method of moments).
The latter is needed mainly because the calculation of the moments is simplified and less cumbersome. As remarked above, it is needed for the MSPE method only for $k=2$ (the same condition that appears in \cite{kluppelbergEJ2007}), while to calculate the OPBE  it is asked for $k=4$. Let also remark  that in real financial data it is known that often there is an asymmetric response of the volatility and this assumption is  not able to model this leverage effect, (see also a comment about that in \cite{kluppelbergEJ2007}). An improved version of the COGARCH model that can take in account such effect has been recently proposed in \cite{asimmetrico}.
Still, however there is a number of parametric families of L\'evy processes like Compound Poisson, Normal Inverse Gaussian, Variance Gamma and Meixner processes  that for some value of their parameters satisfy both Conditions \ref{final_cond_MSPE} and Condition \ref{further_conditions_OPBE}. A detailed discussion is presented for the Variance Gamma family below.

\subsection{Example: Variance Gamma}

The Variance Gamma process $V_{t}$ is an infinity activity pure jump L\'evy process that has been used itself to model log returns \cite{madanSeneta}. 
The characteristic function is given by
$$
\E\left(e^{iuV_t}\right)=\left( 1+ \frac{A^{2}u^2}{2C} \right)^{-tC}, \quad C>0,A>0.
$$ 
and the L\'evy measure has density 
\begin{equation}
\nu_L(dx)=\frac{C}{|x|} \exp\left( -\frac{|x|}{A} \sqrt{2C}\right) dx \quad x\neq 0.
\label{nu}\end{equation}
The Variance Gamma process has finite moments of any order and a symmetric density which cannot be expressed in a closed form. Its variance is given by $A^{2}t$.
If we assume that it drives (without a Brownian component) a COGARCH(1,1) model, the first of Conditions \ref{final_cond_MSPE} imposes $A=1$, while the parameter $C$ remains free.
%Moreover if we consider a driving L\'evy Process $L_{t}=\tau B_{t}+V_{t}$ with a Brownian component we get from the first of Conditions \ref{final_cond_MSPE} that the parameters $A$ and $\tau$ are related by $A^{2}=1-\tau^{2}$.

\section{Numerical example}\label{simulation_study}
To evaluate the performance of PBEF compared with other estimators we set up a numerical example. 
The setting is the same as in \cite{kluppelbergEJ2007}: the model is a COGARCH(1,1) with true parameters $\theta_{0}=(\beta_{0},\eta_{0},\phi_{0})=(0.04,0.053,0.038)$ driven by a Variance Gamma process. The parameter $C$ in \eqref{nu} is fixed to $C=1$ (and non estimated). In this setting we have $\Psi(4)=-0.0261<0$ and Condition \ref{final_cond_MSPE} is fulfilled. The code for the numerical example (simulation, evaluation of the moments and estimation with the different methods) has been made available as an R package called COGARCH. Instructions on how to install and use it are included as Supporting Information.

\subsection{Asymptotic variances}
Numerical evaluation of the iterative expressions of Section \ref{moment} for the higher moments makes possible to calculate the asymptotic variances of the MSPE estimator obtained with the weights in \eqref{naive} and of the OPBE (Optimal Prediction Based Estimator, see \eqref{optweights1}).
The explicit calculation of the matrix $M_{n}(\theta_{0})$ in \eqref{matriceMn} takes a long time (nearly 100 minutes on a recent personal computer, with a careful C implementation) and is not feasible for repeated evaluation within an optimization algorithm. Since the sequence $M_{n}(\theta_{0})$ of equation \eqref{matriceMn} converges exponentially fast to $M(\theta_{0})$ and since our numerical experiment actually demonstrates that already $M_{0}(\theta_{0})$ approximates $M(\theta_{0})$ very well, it seems reasonable to use $M_{0}(\theta_{0})$ instead of $M_{n}(\theta_{0})$ in the weights \eqref{optweights1} and to call this estimator \emph{approximate} OPBE. With such approximation, a negligible increase in the asymptotic variance is introduced. The results are summarized in Table \ref{AV}.
\begin{table}
\begin{center}
\begin{tabular}{cc}
MSPE&$\displaystyle\begin{pmatrix}  4.668 &2.989& 1.216\\
2.989& 3.172& 2.058\\
1.216& 2.058& 1.628\end{pmatrix}$\\
\phantom{M}&\\
OPBE&$\displaystyle\begin{pmatrix} 4.503& 2.844& 1.133\\
2.844& 3.045& 1.985\\
1.133& 1.985& 1.587\end{pmatrix}$\\
\phantom{M}&\\
approximate OPBE&$\displaystyle\begin{pmatrix} 4.504& 2.845& 1.134\\
2.845& 3.047& 1.988\\
1.134& 1.988& 1.588\end{pmatrix}$
\end{tabular}
\end{center}
\caption{Asymptotic variances $V(\theta_{0})$ of $\sqrt{n}\,(\hat\theta-\theta_{0})$ given by Theorem \ref{asymptotics}.\label{AV}}
\end{table}

\subsection{Simulation study}
To investigate the finite sample properties of the different estimation methods we set up a simulation study.
We generate 10000 trajectories of $G$ with 20000 observations separated by a time lag $r=1$, $\{G_{j}\}_{j=1\ldots 20000}$ and computed the log-returns $G_{j, 1}$. To increase the accuracy of the simulation the actual grid for the Euler method was 1000 times finer with respect to the final grid of the observations.
From each simulated sample we estimated the parameters with all the four available methods: the method of moments (MME) of \cite{kluppelbergEJ2007}, the Pseudo-Maximum Likelihood (PML) method proposed in \cite{mleMaller}, and the new estimators MSPE and OPBE introduced above. 
Aware of the results of the previous subsection we calculated the OPBE approximatively using $M_{0}(\theta)$ instead of $M_{n}(\theta)$ in the weights.

\begin{table}
\begin{center}
\begin{tabular}{ccccc}
method&$\text{avg}(\hat{\theta})$&$\theta_{0}$&$\frac{|\text{avg}(\hat{\theta})-\theta_{0}|}{\theta_{0}}$&$\text{cov}(\hat{\theta})$\\\hline
\phantom{M}&\\
MME&
%medie
$\displaystyle\begin{matrix}0.0498\\
0.0585\\
0.0397\end{matrix}$
&
%valori veri
$\displaystyle\begin{matrix}0.04\\
0.053\\
0.038\end{matrix}$
&
%bias
$\displaystyle\begin{matrix}0.245\\
0.103\\
0.045\end{matrix}$
&
%cov
$\displaystyle\begin{pmatrix}2.41\text{e}-04 & 1.79\text{e}-04 & 8.63\text{e}-05 \\ 
  1.79\text{e}-04 & 1.81\text{e}-04 & 1.14\text{e}-04 \\ 
  8.63\text{e}-05 & 1.14\text{e}-04 & 8.23\text{e}-05\end{pmatrix}$ \\
\phantom{M}&\\
PML&
%medie
$\displaystyle\begin{matrix}0.0418\\
0.0433\\
0.0278\end{matrix}$
&
%valori veri
$\displaystyle\begin{matrix}0.04\\
0.053\\
0.038\end{matrix}$
&
%bias
$\displaystyle\begin{matrix}0.046\\
0.184\\
0.268\end{matrix}$
&
%cov
$\displaystyle\begin{pmatrix}    6.85\text{e}-05 & 4.12\text{e}-05 & 1.49\text{e}-05 \\ 
  4.12\text{e}-05 & 3.19\text{e}-05 & 1.65\text{e}-05 \\ 
  1.49\text{e}-05 & 1.65\text{e}-05 & 1.19\text{e}-05 \end{pmatrix}$\\
\phantom{M}&\\
MSPE&
%medie
$\displaystyle\begin{matrix}0.0430\\
0.0538\\
0.0376\end{matrix}$
&
%valori veri
$\displaystyle\begin{matrix}0.04\\
0.053\\
0.038\end{matrix}$

&
%bias
$\displaystyle\begin{matrix}0.075\\
0.014\\
0.011\end{matrix}$
&
%cov
$\displaystyle\begin{pmatrix} 1.57\text{e}-04 & 1.18\text{e}-04 & 5.87\text{e}-05 \\ 
  1.18\text{e}-04 & 1.15\text{e}-04 & 7.14\text{e}-05 \\ 
  5.87\text{e}-05 & 7.14\text{e}-05 & 5.06\text{e}-05\end{pmatrix}$\\
\phantom{M}&&\\
OPBE&
%medie
$\displaystyle\begin{matrix}0.0428\\
0.0533\\
0.0372\end{matrix}$&
%valori veri
$\displaystyle\begin{matrix}0.04\\
0.053\\
0.038\end{matrix}$
&
%bias
$\displaystyle\begin{matrix}0.069\\
0.006\\
0.021\end{matrix}$
&
%cov
$\displaystyle\begin{pmatrix} 1.34\text{e}-04 & 9.86\text{e}-05 & 4.75\text{e}-05 \\ 
  9.86\text{e}-05 & 9.52\text{e}-05 & 5.81\text{e}-05 \\ 
  4.75\text{e}-05 & 5.81\text{e}-05 & 4.11\text{e}-05\end{pmatrix}$
\end{tabular}
\end{center}
\caption{Summary statistics of the estimates.\label{summary} The values of the OPBE are approximated as described in the text.}
\end{table}

Descriptive statistics of the estimates are compared in Table \ref{summary}. Figure \ref{densities} illustrates the empirical densities of the estimates for the three parameters. 
The OPBE outperforms all the other methods. Indeed, PML estimator has a smaller variance, but a relevant bias, while the MME has an higher variance. The difference between the OPBE and the suboptimal MSPE is very small (and difficult to be appreciated in Figure \ref{densities}).
Let us remark that although our sample is very large (20000 observations at lag 1 for each trajectory) asymptotic normality is far from being reached as Figure \ref{QQ} shows for both MME and OPBE (we plotted the QQ plot just for the parameter $\eta$ since the other two are very similar). A possible reason is that asymptotic normality only holds under $\Psi(4)<0$: in the chosen parameter setting the value of $\Psi(4)$ is negative but very close to zero. The heavy tail of the estimates may also explain the fact that the the empirical covariances displayed in Table \ref{summary} sistematically underestimate their theoretical value that can be obtained dividing by the sample size $n=20000$ the values reported in Table \ref{AV}.

\begin{figure}[h]
   \centering
   \includegraphics[width=1.00\textwidth]{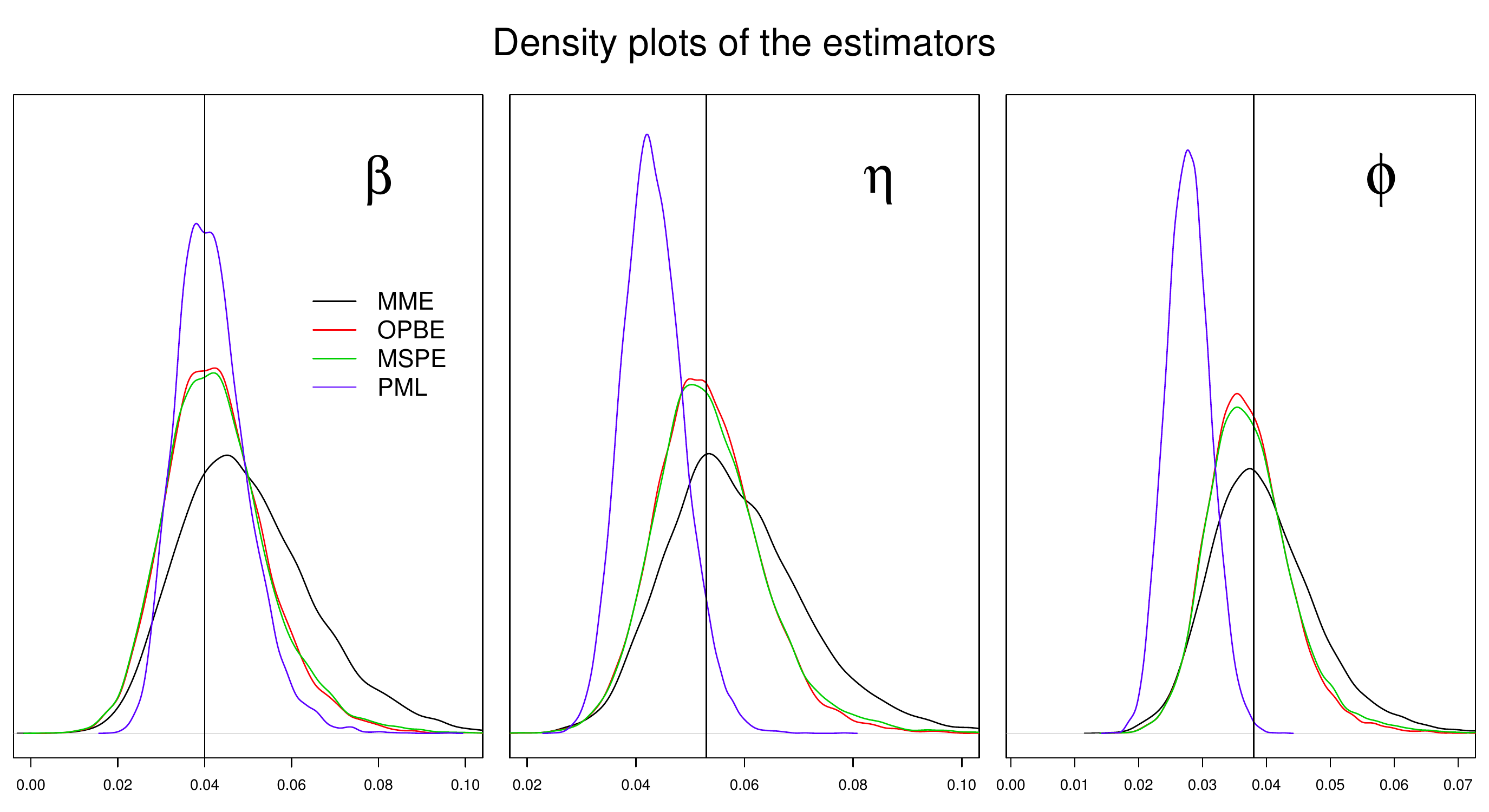}
   \caption{Comparison between the empirical densities of the different estimators.\label{densities}}
\end{figure}

\begin{figure}[h]
   \centering
   \includegraphics[width=1.00\textwidth]{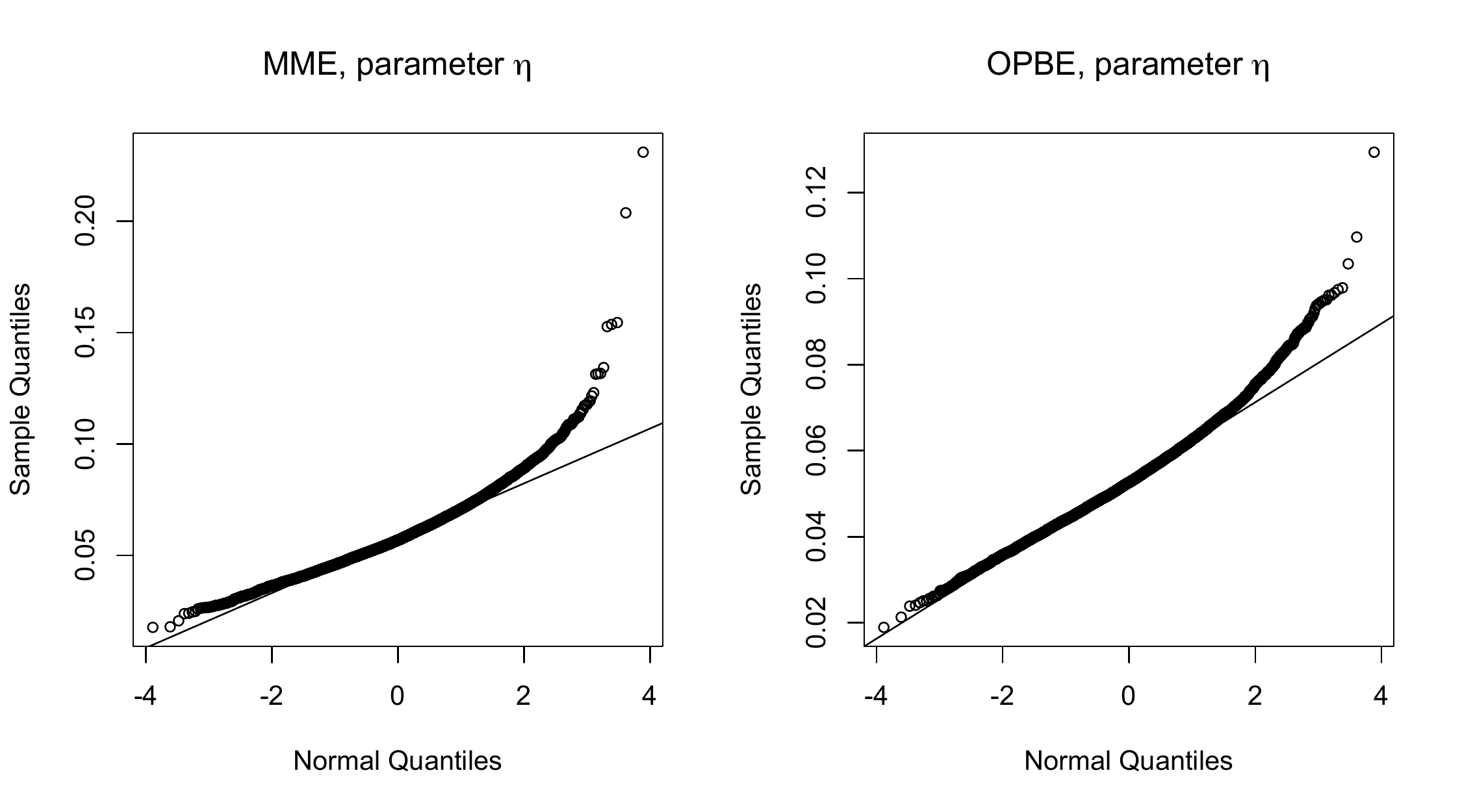}
   \caption{QQplot of the MME and OPBE of parameter $\eta$.\label{QQ}}
\end{figure}

\section{Conclusion and further developments}
We specialized the method of Prediction Based Estimating function for the COGARCH(1,1) model. Motivated by the search for an optimal PBEF, we have investigated the higher order structure of the process. Iterative expressions to calculate the higher order moments are derived. The asymptotic properties of the estimators are studied and illustrated by a numerical example. PBEF are shown to outperform all other available estimation methods. Further work will be dedicated to the development of an R package to simulate and estimate the parameters of the COGARCH model. A faster algorithm to calculate the matrix $M_{n}(\theta)$ would also be desirable. There is no technical obstruction to apply the PBEF method to non equally spaced observations, a larger empirical study would be worth.

\label{conclusion}

\section*{Supporting information}\label{supp}
Additional information for this article is available online. It regards the R package COGARCH that provides the code for the simulation, estimation, and evaluation of the moments and also a Mathematica code that implements the recursive formulae of Section \ref{moment}. Due to their length, indeed, the outputs of such formulae are unmanageable for a human. We include as a supplementary material to this paper a Mathematica notebook that allows to calculate all of them and to provide the results as computer-readable text files that are then used into the R package. The content of the supplementary material is the following:
\begin{itemize}
\item the file \emph{how2useCOGARCH.pdf} that is a brief documentation of the R package COGARCH
\item the file \emph{ExplicitExpressions.nb} is the Mathematica notebook \cite{Mathematica}
\item the file \emph{ExplicitExpressions.pdf} reports the content of the notebook in a file format that is readable without a Mathematica license
\item a directory where precomputed outputs of the notebook are stored in order to make the execution of the notebook faster. The notebook does not need the presence of such directory, however in its absence it has to compute all the expression anew and the execution becomes much slower. Moreover into the subfolder R the final expressions of the moments are saved as text files with a syntax that is compatible with the open source software R \cite{R}.
\end{itemize}

\section*{Acknowledgements}
This work has been supported by PRIN 2009JW2STY. The authors are grateful to Claudia Kluppelberg for her stimulating encouragement and to Michael S{\o}rensen and Susanne Ditlevsen, that both suggested the use of the approximation for the limit matrix. Thanks to Giovanni Birolo who wrote the  supplementary material. EB dedicates this paper to Martino and fagiolino.

\bibliographystyle{apalike}
\bibliography{BLSS}

\end{document}